\documentclass[journal]{IEEEtran}
\usepackage{amssymb,epsfig,color,cite,amsmath,amsfonts,mathrsfs,algorithm,algorithmicx}
\usepackage{epstopdf}
\usepackage{datetime,fancyhdr}
\usepackage{algpseudocode}
\usepackage{arydshln}
\usepackage{multirow}
\usepackage{amsthm} 

\usepackage{times} 
\usepackage{epsfig}
\usepackage{graphicx}
\usepackage{latexsym}
\usepackage{subfigure}
\allowdisplaybreaks
\newtheorem{lemma}{Lemma}[section]
\newtheorem{theorem}{Theorem}[section]
\newtheorem{corollary}{Corollary}[section]
\newtheorem{remark}{Remark}[section]
\newtheorem{definition}{Definition}[section]

\newtheorem{assumption}{Assumption}[section]

\usepackage{hyperref}
\usepackage{caption}
\ifCLASSINFOpdf
\else
\fi
\begin{document}
%
\title{Shuffling Gradient Descent-Ascent  with Variance Reduction for Nonconvex-Strongly Concave Smooth Minimax Problems}
%
%
%

\author{Xia Jiang, Linglingzhi Zhu, Anthony Man-Cho So, Shisheng Cui, Jian Sun

\thanks{X. Jiang (xiajiang@cuhk.edu.hk), L. Zhu (llzzhu@se.cuhk.edu.hk) and A. M.-C. So (manchoso@se.cuhk.edu.hk) are with Department of Systems Engineering and Engineering Management, The Chinese University of Hong Kong, Hong Kong, 999077.
}
\thanks{J. Sun (sunjian@bit.edu.cn) and S. Cui (css@bit.edu.cn) are with the National Key Laboratory of Autonomous Intelligent Unmanned Systems, School of Automation, Beijing Institute of Technology, Beijing 10081, China.}
}
\maketitle

\begin{abstract}
In recent years, there has been considerable interest in designing stochastic first-order algorithms to tackle finite-sum smooth minimax problems. To obtain the gradient estimates, one typically rely on the uniform sampling-with-replacement scheme or various sampling-without-replacement (also known as shuffling) schemes. While the former is easier to analyze, the latter often have better empirical performance. In this paper, we propose a novel single-loop stochastic gradient descent-ascent (GDA) algorithm that employs both shuffling schemes and variance reduction to solve nonconvex-strongly concave smooth minimax problems. We show that the proposed algorithm achieves $\epsilon$-stationarity in expectation in $\mathcal{O}(\kappa^2 \epsilon^{-2})$ iterations, where $\kappa$ is the condition number of the problem. This outperforms existing shuffling schemes and matches the complexity of the best-known sampling-with-replacement algorithms. Our proposed algorithm also achieves the same complexity as that of its deterministic counterpart, the two-timescale GDA algorithm. Our numerical experiments demonstrate the superior performance of the proposed algorithm.
\end{abstract}


%
\IEEEpeerreviewmaketitle




\section{Introduction}
%
%
%
%
\IEEEPARstart{M}{inimax} optimization problems have attracted widespread interest in recent years due to their diverse applications in various fields, including adversarial learning \cite{NIPS2014_GAN}, reinforcement learning \cite{RL_saddle}, robust controller design \cite{robust_control}, resource allocation \cite{source_allo_conti} and model predictive control \cite{mini_model_pre}. While the ultimate objective is to train models that perform well to unseen data, in practice, we deal with a finite dataset during training. This leads to the finite-sum minimax optimization problem addressed in this paper:
\begin{align}\tag{P}\label{opti_pro}
\min_{x\in \mathbb{R}^{d}}\max_{y\in \mathbb{R}^d} f(x,y):=\frac{1}{n} \sum_{i=1}^{n} f_i(x,y),
\end{align}
where $f_i:\mathbb{R}^d\times \mathbb{R}^d\rightarrow \mathbb{R}$ denotes the smooth loss function associated with the $i$-th data sample. 
Unlike existing studies on convex-concave minimax problems \cite{unified_dis_tac, pmlr-v89-du19b,DPD_aggre_game}, we focus on solving a nonconvex case, where \( f \) is strongly concave in \( y \) but nonconvex in \( x \). 
In the application of power control in wireless communication \cite{Lu_2020}, involving a jammer in a multi-user and multi-channel interference channel transmission, the maximization jammer problem exhibits a strongly concave objective function. Then, the power control problem transforms into a nonconvex-strongly concave minimax problem. In reinforcement learning, \cite{pmlr-v80-dai18c} presents that the Bellman optimality equation can be reformulated into a nonconvex-strongly concave primal-dual optimization problem using Nesterov’s smoothing technique and the Legendre-Fenchel transformation. More applications can be found in data poisoning \cite{pmlr-v119-liu20j} and distributionally robust optimization \cite{SREDA_nips2020}. 

For nonconvex minimax problems with a finite-sum structure, stochastic first-order algorithms have garnered considerable research interest due to their scalability and efficiency. The most commonly used one is the stochastic gradient descent-ascent (SGDA) algorithm, a natural extension of stochastic gradient descent algorithm for minimax problems. In nonconvex-strongly concave (NC-SC) case, \cite{SGDA} proposed a stochastic variant of the two-timescale GDA algorithm \cite{two_GDA}. This stochastic variant incorporates unequal step sizes and has established a non-asymptotic convergence result. Another approach is the multi-step stochastic gradient-based algorithms \cite{SGD_max, SREDA_nips2020, NC_if}, which involves 
an inner loop seeking an approximate solution for $\max_{y\in \mathbb{R}^d} f(x,y)$ for  given $x$ and an outer loop as the inexact gradient descent step on $x$. Both of these algorithms emphasize the importance of updating $y$ more frequently than $x$ for solving minimax problems, while the two-timescale SGDA is relatively easier to implement and generally demonstrates superior empirical performance compared to the multi-step algorithms.

While stochastic gradient methods often assume uniform sampling with replacement to achieve unbiased gradient estimates and simplify theoretical analysis, practical implementations frequently deviate from this ideal scenario. Instead, various heuristics are incorporated to enhance the algorithm's performance based on empirical findings. One common and notable heuristic is the sampling-without-replacement scheme (e.g., incremental, shuffling once, and random reshuffling), known as shuffling methods, which learn from each data point in every epoch and generally own better performance than SGD \cite{why_rand_good}. In PyTorch and TensorFlow, the sampling-without-replacement random reshuffling scheme has been widely adopted in machine learning training. Although shuffling methods demonstrate excellent practical performance, the presence of biased gradients poses challenges in establishing theoretical convergence guarantees. Until recent years, the empirical and theoretical benefits of the sampling-without-replacement scheme have received considerable attention for minimization problems (e.g., \cite{RR_improv, SGD_opti_rate, why_rand_good}), however, studies on its application to minimax problems are relatively scarce. The primal-dual balancing in minimax problems introduces additional challenges due to the absence of unbiased gradient estimates, especially in the context of nonconvex minimax problems.



In this paper, we focus on designing shuffling single-loop  algorithms with practical explicit updating schemes to solve nonconvex minimax problems. Notably, recent work \cite{SGDA_ICLR_NCPL} has developed SGDA with random reshuffling sampling (RR), one of the sampling-without-replacement schemes, for solving smooth nonconvex-Polyak-Łojasiewicz (an extension of the nonconvex-strongly concave case by including dual degeneracy) finite-sum minimax problems. This approach achieved a gradient complexity of $\mathcal{O}(\sqrt{n}\kappa^{3}\epsilon^{-3})$, better than the $\mathcal{O}(\kappa^{3}\epsilon^{-4})$ complexity of SGDA when $\epsilon\leq \mathcal{O}(1/\sqrt{n})$, specifically for solving stochastic nonconvex-strongly concave minimax problems. However, it remains significantly slower than the prototype deterministic algorithm, i.e., two-timescale GDA, which has a gradient complexity of $\mathcal{O}(n\kappa^{2}\epsilon^{-2})$.
Consequently, one natural question arises: 
\par \textit{
Is it possible to improve the convergence rate of single-loop sampling-without-replacement stochastic algorithms for nonconvex minimax problems to match their deterministic counterparts?
} 



We affirmatively address this question by utilizing variance reduction techniques. As a standard tool to improve the efficiency of stochastic optimization algorithms, several existing variance-reduced stochastic algorithms have been developed to enhance convergence performance by mitigating gradient variance in various settings of minimax optimization \cite{VR_CC, GAN_VR_nips, comp_vr_finite, chen2023faster, NIPS_VR_PL, NIPS_SAPD, vr_vi_cui, SREDA_nips2020,VR_mon_general}. However, this is the first time a suitable variance reduction scheme has been considered and designed specifically to improve the performance of shuffling algorithms for nonconvex minimax problems.
The contributions of this paper are summarized as follows. 

\begin{itemize}
    \item We develop a novel single-loop uniform shuffling (encompassing incremental, shuffling once, and random reshuffling) gradient descent-ascent algorithm with variance reduction techniques for acceleration. The variance reduction step eliminates the need for assuming bounded gradient variance and enables our algorithm to converges to the stationary points of nonconvex minimax problems, without introducing the additional computational cost associated with choosing a specific batch size, as required in SGDA \cite{SGDA}.
    \item Our algorithm, with its two-timescale step sizes tailored for NC-SC finite-sum minimax problems, achieves a gradient complexity of $\mathcal{O}(n\kappa^2 \epsilon^{-2})$ for deriving an $\epsilon$-optimization-stationary point.
This complexity result not only outperforms the best-known result of $\mathcal{O}(\sqrt{n}\kappa^{3}\epsilon^{-3})$ for shuffling-type algorithms reported in \cite{SGDA_ICLR_NCPL} if $\epsilon\leq \mathcal{O}(\kappa/\sqrt{n})$, but also matches the  best-known gradient complexity achieved by the multi-step uniform sampling-with-replacement  algorithm SREDA in \cite{SREDA_nips2020}.
  \item Furthermore, our proposed algorithm has the same complexity as its deterministic prototype algorithm, the two-timescale GDA, effectively managing the expense of taking the finite-sum structure without additional costs. This also illustrates the potential of shuffling with variance reduction techniques to improve complexity for a broader range of discrete iterative update schemes, including but not limited to nonconvex minimax problems.
\end{itemize}

The notation we use in this paper is standard. We use $[n]$ to denote the set $\{1,2,\ldots,n\}$ for any positive integer $n$. Let the Euclidean space of all real vectors be equipped with the inner product $\langle x,y\rangle:=x^{\top}y$ for any vectors $x, y$ and denote the induced norm by $\|\cdot\|$. 
For a differentiable function $f$, the gradient of $f$ is denoted as $\nabla f$.

The remainder of this paper is organized as follows. In Section \ref{solver_design}, we present the problem setting and preliminary results for minimax optimization. The proposed shuffling gradient descent-ascent algorithm with variance reduction is introduced in Section \ref{pro_trans}. In Section \ref{proof_sec}, we provide a step-by-step proof of the convergence result for the proposed algorithm. We conclude with numerical experiments demonstrating the empirical effectiveness of our algorithm in Section \ref{simulation}, followed by a brief summary of the work in Section \ref{conclusion}.


\section{Problem Setup and Preliminaries}\label{solver_design}


We introduce the basic NC-SC problem setup and key concepts that will serve as the basis for our subsequent analysis.
\begin{assumption}\label{f_assump}
The following assumptions on the objective function $f$ of problem \eqref{opti_pro} hold throughout the paper. 
\begin{itemize}
    \item[(i)] $f_i$ is $l$-smooth, i.e.,
    \begin{align*}
        \|\nabla f_i(x_1,y_1)\!-\!\nabla f_i(x_2,y_2)\|\!\leq \!l(\|x_1-x_2\|\!+\!\|y_1-y_2\|);
    \end{align*}
        \item[(ii)] $f$ is nonconvex in $x$ but $\mu$-strongly concave in $y$, with  the condition number $\kappa:=l/\mu\geq 1$.
\end{itemize}
\end{assumption}

The minimax problem \eqref{opti_pro} is equivalent to the following minimization problem
\begin{align}\label{min_opti}
    \min_{x\in \mathbb{R}^d} \left\{\Phi(x):= \max_{y\in \mathbb{R}^d} f(x,y)\right\}.
\end{align}
In the NC-SC setting, the function $\Phi$ possesses several important properties as described in the following technical lemma.
\begin{lemma}[{cf. \cite[Lemma 4.3]{SGDA}}]\label{Phi_lem}
Under Assumption \ref{f_assump}, the function $\Phi$ is $(l+\kappa l)$-smooth with $\nabla \Phi(x)=\nabla_x f(x, y^*(x))$, where $y^*(x)\in\operatorname{argmax}_{y\in \mathbb{R}^d} f(\cdot,y)$ is a singleton. Also, $y^*(\cdot)$ is $\kappa$-Lipschitz.
\end{lemma}


Since $f(x,\cdot)$ is concave for each $x \in \mathbb{R}^d$, the maximization problem $\max_{y \in \mathbb{R}^d} f(x,y)$ can be solved efficiently, providing valuable information $y^*(x)$ for the function $\Phi$. However, because $\Phi$ is still nonconvex in $x$, a reasonable candidate for Nash equilibria or global minima in the nonconvex minimax problem \eqref{opti_pro} is to identify the stationary point of $\Phi$ using the commonly used stationarity measure $\|\nabla \Phi(x)\|$ (with differentiability by Lemma \ref{Phi_lem}).


\begin{definition}\label{sta_point_def}
    The point $x$ is an $\epsilon$-optimization-stationary point of problem \eqref{opti_pro} if $\|\nabla \Phi(x)\|\leq \epsilon$ holds with $\epsilon\ge0$.
\end{definition}





\section{Shuffling Gradient Descent-Ascent with Variance Reduction}\label{pro_trans}
In this section, we propose a variance-reduced stochastic algorithm with the shuffling framework (Algorithm \ref{vr_rr}) for solving the NC-SC minimax problem \eqref{opti_pro}. To avoid undesired gradient variance while utilizing the practical benefits of the sampling-without-replacement scheme, the proposed stochastic algorithm incorporates two main techniques: a data shuffling step and a variance reduction step.


First, the proposed algorithm uses a permutation/shuffling of the data $\{1, \ldots, n\}$ that can be random or deterministic, including incremental gradient (IG), shuffling once (SO), and random reshuffling (RR). Among these shuffling schemes, RR is the most common but theoretically elusive one, where a new permutation is generated at the beginning of each epoch. The SO scheme shuffles the dataset only once and then reuses this random permutation in all subsequent epochs. The IG scheme is similar to SO, except that the initial permutation is deterministic rather than random. The ordering in IG could be arbitrary, such as the order in which the data arrives or an adversarially chosen sequence. Our primary focus is on the RR scheme, with the SO and IG schemes being analyzed as byproducts of our analysis of the RR scheme.


In the meantime, variance-reduced gradient estimation is performed. Following the order in the sample sequence $\pi$, we sequentially conduct $n$ iterative primal and dual updates using the corresponding data samples. In each iteration $j$, the gradient estimate is updated using a simple variance reduction technique similar to that introduced in \cite{SVRG_nips}. The full gradient is updated after processing all $n$ data samples. 

\begin{algorithm}
\caption{Shuffling Gradient Descent-Ascent with Variance Reduction}
	\label{vr_rr}
	\begin{algorithmic}[0]  
        \State \textbf{Initialization}: $x_0, y_0\in \mathbb{R}^{d}$, step sizes $\eta_1, \eta_2>0$, number of epoches $T$, deterministic (resp. random) permutation $\pi=(\pi^0,\ldots,\pi^{n-1})$ of $[n]$ for IG (resp. SO).
        \For {$t=0,1,\ldots,T$}
        \State Randomly permute $[n]$ to $\pi = (\pi^0, \ldots, \pi^{n-1})$ for RR
        \State $x_{t}^{0}=x_{t}$, $y_{t}^{0}=y_t$
        \State $h_t^0=\frac{1}{n}\sum_{i=1}^n \nabla_x f_i(x_t,y_t)$
        \State $d_t^0=\frac{1}{n}\sum_{i=1}^n \nabla_y f_i(x_t,y_t)$
        \For {$j=0,\ldots,n-1$}
        \State $h_{t}^j=h_t^0+\nabla_x f_{\pi^j}(x_{t}^{j},y_t^j)-\nabla_x f_{\pi^j}(x_t,y_t)$
        \State $x_{t}^{j+1}=x_{t}^{j}-\frac{\eta_1}{n} h_{t}^j$
        \State $d_{t}^{j}=d_t^0+\nabla_y f_{\pi^j}(x_t^j,y_t^j)-\nabla_y f_{\pi^j}(x_t,y_t)$
        \State $y_{t}^{j+1}=y_{t}^{j}+\frac{\eta_2}{n} d_{t}^{j}$
        \EndFor
        \State $x_{t+1}=x_{t}^{n}$, $y_{t+1}=y_{t}^{n}$
        \EndFor
	\end{algorithmic}
\end{algorithm}

\begin{remark}
Compared to the widely used uniform sampling procedure with replacement, the shuffling procedure offers a significant advantage by avoiding cache misses during data retrieval. The efficiency of shuffling schemes stems from their ability to learn from each sample in every epoch, ensuring that no sample is missed. In contrast, the sampling-with-replacement scheme may overlook certain samples in a given epoch. Our proposed algorithm for solving minimax optimization problems retains this beneficial property of shuffling algorithms. It is a non-nested loop shuffling-friendly algorithm that does not depend on the inner solver for an accurate solution to the dual problem, requiring only an additional cost for full gradient updates every $n$ steps.

\end{remark}

\section{Convergence analysis}\label{proof_sec}
We first reformulate the loss function $f$ in \eqref{opti_pro} into an equivalent form. Let auxiliary vectors $a_1,\ldots,a_n\in \mathbb{R}^{d}$ and vectors $b_1,\ldots,b_n\in \mathbb{R}^{d}$ satisfy the conditions $\sum_{i=1}^n a_i=0$ and $\sum_{i=1}^n b_i=0$. 
Then we can reformulate the loss function $f$ into the following equivalent form 
\begin{align}\label{trans_f}
f(x,y)
=\frac{1}{n}\sum_{i=1}^n \tilde{f}_i(x,y),
\end{align}
where $\tilde{f}_i(x,y):= f_i(x,y)+\langle a_i,x\rangle+\langle b_i,y\rangle$. For the proposed Algorithm \ref{vr_rr} with any $t\in[T]$, let $a_i=\nabla_x f(x_t,y_t)-\nabla_x f_i(x_t,y_t)$ and $b_i=\nabla_y f(x_t,y_t)-\nabla_y f_i(x_t,y_t)$. Then, the updates of $h_{t}^j$ and $d_{t}^j$ can be rewritten as
\begin{center}
    $h_{t}^j=\nabla_x \tilde{f}_{\pi^j}(x_t^j,y_t^j)$,\quad
    $d_{t}^j=\nabla_y \tilde{f}_{\pi^j}(x_t^j,y_t^j)$.
\end{center}
Also, we know that the  updating scheme of the proposed Algorithm \ref{vr_rr} can be rewritten as 
\begin{center}
$x_{t+1} = x_t -\eta_1 h_t$,\quad $y_{t+1} = y_t +\eta_2 d_t$,
\end{center}
where $h_t:=\frac{1}{n} \sum_{j=0}^{n-1}h_t^j=\frac{1}{n} \sum_{j=0}^{n-1} \nabla_x \tilde{f}_j(x_t^j, y_t^j)$ and $d_t:=\frac{1}{n} \sum_{j=0}^{n-1}d_t^j=\frac{1}{n} \sum_{j=0}^{n-1} \nabla_y \tilde{f}_j(x_t^j, y_t^j)$.

In the following subsections, to derive the convergence properties of Algorithm \ref{vr_rr},  we introduce a potential function $\mathcal{P}_{\lambda}:\mathbb{R}^d\times \mathbb{R}^d\rightarrow \mathbb{R}$ with given $\lambda>0$ satisfying
$$\mathcal{P}_{\lambda}(x,y):=\lambda (\Phi(x)-\Phi^*)+\Phi(x)-f(x,y),$$
which is commonly used for convergence analysis in nonconvex-strongly concave minimax optimization; see \cite{NIPS_VR_PL,SGDA_ICLR_NCPL}.
Since  for any $(x,y)$ it has $\Phi^*\leq \Phi(x)$ and $f(x,y)\leq \Phi(x)$, the potential function $\mathcal{P}_{\lambda}$ is lower bounded by $0$.
We first establish the basic descent property of the potential function, and then control the bias in the shuffling algorithm, which is the backward per-epoch deviation. With these results, we are finally able to provide the convergence rate analysis of the proposed algorithm.

\subsection{Basic Descent Properties}

Under the sampling-without-replacement scheme, we denote the backward per-epoch deviations  with $z_t:=[x_t;y_t]$ and $z_t^j:=[x_t^j;y_t^j]$ at each epoch $t$ as
\begin{equation*}
  B_t := \sum_{j=0}^{n-1} \|z_t - z_t^j\|^2=\sum_{j=0}^{n-1} (\|x_t - x_t^j\|^2+\|y_t - y_t^j\|^2).
\end{equation*}
Utilizing the structure of the function $\mathcal{P}_{\lambda}$, we establish the difference in function value for $x_{t}$ to $x_{t+1}$ in Algorithm \ref{vr_rr}.

\begin{lemma}\label{Phi_dif_lem}
Under Assumption \ref{f_assump}, it follows that
\begin{align}\label{Phi_ineq}
     &\mathcal{P}_{\lambda}(x_{t+1},y_{t+1})-\mathcal{P}_{\lambda}(x_t,y_t)\notag\\
     \leq &-\frac{(\lambda+1)\eta_1}{2}\|\nabla \Phi(x_t)\|^2\! +2(\lambda+1)\eta_1 l\kappa  [\Phi(x_t)-\!f(x_t,y_t)] \notag\\
    &+\frac{\eta_1}{2}\|\nabla_x f(x_t,y_t)\|^2-\frac{\eta_2}{2}\|\nabla_y f(x_t,y_t)\|^2
    \notag\\
    &+\left(\lambda \eta_1+\frac{\eta_1+\eta_2}{2}\right)\frac{l^2}{n}B_t\notag\\
    &+\left(\frac{\eta_1}{2}(1+\eta_1 l)-\frac{(\lambda+1)\eta_1}{2}\cdot\left(1-(\kappa l+l)\eta_1 \right)\right)\left\|h_t\right\|^2\notag\\
    &-\frac{\eta_2}{2}(1-\eta_2 l)\|d_t\|^2
\end{align}
\end{lemma}
\begin{proof}
\par Since $\Phi$ is $(l+\kappa l)$-smooth by Lemma \ref{Phi_lem},
\begin{align}\label{phi_diff}
&\Phi(x_{t+1})-\Phi(x_{t})\notag\\
\leq\ & \nabla \Phi(x_{t})^\top(x_{t+1}-x_{t}) + \frac{\kappa l+l}{2} \|x_{t+1}-x_{t}\|^2\notag\\
    =\ &-\eta_1\nabla \Phi(x_t)^\top  h_t+ \frac{\kappa l+l}{2} \|\eta_1 h_t\|^2\notag\\
    =\ &- \frac{\eta_1}{2} \left(\|\nabla \Phi(x_t)\|^2+\|h_t\|^2-\|\nabla \Phi(x_t)-h_t\|^2\right)\notag\\
    &+\frac{(\kappa l+l)\eta_1^2}{2} \|h_t\|^2\notag\\
    =&- \frac{\eta_1}{2}\|\nabla \Phi(x_t)\|^2+\frac{\eta_1}{2} \|\nabla \Phi(x_t)-h_t\|^2\notag\\
    &-\frac{\eta_1}{2}(1-(\kappa l+l)\eta_1 )\|h_t\|^2.
\end{align}
On the other hand, since  $\nabla \Phi(x_t)=\nabla_x f(x_t,y^*(x_t))$ and $h_t=\frac{1}{n}\sum_{j=0}^{n-1} \nabla_x \tilde{f}_j(x_t^j,y_t^j)=\frac{1}{n}\sum_{j=0}^{n-1}\nabla_x f_j(x_t^j,y_t^j)$ hold, the  term $\left\|\nabla \Phi(x_t)-h_t\right\|^2$ satisfies
\begin{align*}
    &\left\|\nabla \Phi(x_t)-h_t\right\|^2\\
    =\ &\left\| \frac{1}{n}\sum_{j=0}^{n-1} \nabla_x f_j(x_t,y^*(x_t))-\frac{1}{n}\sum_{j=0}^{n-1} \nabla_x f_j(x_t,y_t)\right.\\
    \quad &\left. +\frac{1}{n}\sum_{j=0}^{n-1} \nabla_x f_j(x_t,y_t)-\frac{1}{n}\sum_{j=0}^{n-1}\nabla_x f_j (x_t^j,y_t^j) \right\|^2\\
    \leq\ & 2l^2  \|y^*(x_t)-y_t\|^2+2  \|\nabla_x f(x_t,y_t)-h_t\|^2\\
    \leq\ &4l\kappa  [\Phi(x_t)-f(x_t,y_t)]+2  \|\nabla_x f(x_t,y_t)-h_t\|^2,
\end{align*}
where the last inequality is from the $\mu$-strongly concavity of $f(x,\cdot)$ (i.e. $l^2  \|y^*(x_t)-y_t\|^2\leq 2l\kappa [\Phi(x_t)-f(x_t,y_t)]$).
Substituting the above inequality into \eqref{phi_diff} yields
\begin{align}\label{phi_diff-2}
    &\Phi(x_{t+1})- \Phi(x_t)\notag\\
    \leq\, &- \frac{\eta_1}{2}\|\nabla \Phi(x_t)\|^2-\frac{\eta_1}{2}(1-(\kappa l+l)\eta_1 )\|h_t\|^2\notag\\
    & +2\eta_1 l\kappa  [\Phi(x_t)-f(x_t,y_t)]+\eta_1 \|h_t-\nabla_x f(x_t,y_t)\|^2.
\end{align}   
Next, applying the $l$-smoothness of $f$ yields an upper bound of $f(x_t,y_t)-f(x_{t+1},y_{t+1})$ that
\begin{align*}
& f(x_t,y_t)-f(x_{t+1},y_{t+1}) \\
\leq\ & -\langle\nabla_x f(x_t,y_t), x_{t+1}-x_t\rangle-\langle\nabla_y f(x_t,y_t), y_{t+1}-y_t\rangle\\
&+\frac{l}{2}\|x_{t+1}-x_t\|^2+\frac{l}{2}\|y_{t+1}-y_t\|^2 \\
=\ &\frac{\eta_1}{2}\|\nabla_x f(x_t,y_t)\|^2-\frac{\eta_1}{2}\|h_t-\nabla_x f(x_t,y_t)\|^2\\
&+\frac{\eta_1}{2}(1+\eta_1 l)\|h_t\|^2-\frac{\eta_2}{2}\|\nabla_y f(x_t,y_t)\|^2\\
&+\frac{\eta_2}{2}\|d_t-\nabla_y f(x_t,y_t)\|^2-\frac{\eta_2}{2}(1-\eta_2 l)\|d_t\|^2.
\end{align*}
In addition, by the $l$-smoothness of $f_i$, we know $\|h_t-\nabla_x f(x_t,y_t)\|^2\leq \frac{l^2}{n}B_t$ and $\|d_t-\nabla_y f(x_t,y_t)\|^2\leq \frac{l^2}{n}B_t$.
Combining the above results and \eqref{phi_diff-2}, we know that
\begin{align*}
     &\mathcal{P}_{\lambda}(x_{t+1},y_{t+1})-\mathcal{P}_{\lambda}(x_t,y_t)\\
     =&(\lambda+1)[\Phi(x_{t+1})-\Phi(x_t)]+f(x_t,y_t)-f(x_{t+1},y_{t+1})\\
     \leq&-\frac{\eta_1(\lambda \!+\! 1)}{2}\|\nabla \Phi(x_t)\|^2+2(\lambda \! +\! 1)\eta_1 l\kappa  [\Phi(x_t)-\! f(x_t,y_t)] \\
    &+\frac{\eta_1}{2}\|\nabla_x f(x_t,y_t)\|^2-\frac{\eta_2}{2}\|\nabla_y f(x_t,y_t)\|^2\\
    &+\left(\lambda\eta_1+\frac{\eta_1+\eta_2}{2}\right)\frac{l^2}{n}B_t\\
    &+\left(\frac{\eta_1}{2}(1+\eta_1 l)-\frac{(\lambda+1)\eta_1}{2}\cdot\left(1-(\kappa l+l)\eta_1 \right)\right)\left\|h_t\right\|^2\\
    &-\frac{\eta_2}{2}(1-\eta_2 l)\|d_t\|^2,
\end{align*}
which completes the proof. 
\end{proof}
\subsection{Controlling Bias in Shuffling Algorithms}
In this part, we will further analyze the results derived in Lemma \ref{Phi_dif_lem} and mainly focus on handling the biased term. The following lemma provides the upper bounds for the backward per-epoch deviation $B_t$ in the sampling-without-replacement scheme.
\begin{lemma}\label{upp_bound_B}
Suppose Assumption \ref{f_assump} holds. If the step sizes satisfy $\eta_1^2+\eta_2^2\leq \frac{1}{4l^2}$, then
    \begin{align*}
        \mathbb{E}_t[B_t]
        \leq\ & 4n(\eta_1^2\|\nabla_x f(x_t,y_t)\|^2+\eta_2^2\|\nabla_y f(x_t,y_t)\|^2).
    \end{align*}
\end{lemma}
\begin{proof}
\par By the shuffling gradient descent step in Algorithm \ref{vr_rr}, we know for any $j\in\mathbb{N}$ that
\begin{align}\label{x_sub_i}
    &\|x_t^0-x_t^j\|^2\notag\\
    =\ &\frac{\eta_1^2}{n^2}\left\|\sum_{i=0}^{j-1} 
    \big(\nabla_x f_{\pi^i}(x_t^i,y_t^i)-\nabla_x f_{\pi^i}(x_t,y_t)+\nabla_x f(x_t,y_t)\big)\right\|^2\notag\\
    \leq\ & \frac{\eta_1^2}{n} \left(2l^2 \sum_{i=0}^{n-1}\|z_t^i-z_t\|^2+2  \sum_{i=0}^{n-1}\|\nabla_x f(x_t,y_t)\|^2\right)\notag\\
    =\ & \frac{2\eta_1^2l^2}{n}  B_t+2\eta_1^2 \|\nabla_x f(x_t,y_t)\|^2.
\end{align}
Following the similar proof, we have 
\begin{align}\label{y_sub_i}
    \|y_t^0-y_t^j\|^2\leq \frac{2\eta_2^2l^2}{n}B_t+2\eta_2^2 \|\nabla_y f(x_t,y_t)\|^2.
\end{align}
Summing up \eqref{x_sub_i} and \eqref{y_sub_i},
\begin{align}
    \|z_t^0-z_t^j\|^2
    \leq &\frac{2(\eta_1^2+\eta_2^2)l^2}{n}  B_t+2\eta_1^2 \|\nabla_x f(x_t,y_t)\|^2\notag\\
    &+2\eta_2^2 \|\nabla_y f(x_t,y_t)\|^2
\end{align}
Summing the above inequality over $j=0,\ldots, n-1$ and taking conditional expectation yields
\begin{align*}
    \mathbb{E}_t[B_t]\leq\ &2n(\eta_1^2\|\nabla_x f(x_t,y_t)\|^2+\eta_2^2\|\nabla_y f(x_t,y_t)\|^2)\\
    & + 2(\eta_1^2+\eta_2^2)l^2 \mathbb{E}_t[B_t].
\end{align*}
Thus, when $\eta_1^2+\eta_2^2\leq \frac{1}{4l^2}$, it follows that
\begin{align*}
    \mathbb{E}_t\left[B_t\right]\leq 4n(\eta_1^2\|\nabla_x f(x_t,y_t)\|^2+\eta_2^2\|\nabla_y f(x_t,y_t)\|^2).
\end{align*}
The proof is complete.
\end{proof}

\subsection{Main Convergence Theorem}
Before presenting the convergence theorem, we derive a more concise bound on the per-epoch change of $\mathcal{P}_{\lambda}$ by combining results from Lemmas \ref{Phi_dif_lem} and \ref{upp_bound_B}. This recurrence inequality is essential for  convergence analysis. 

\begin{lemma} \label{rec_lem}
Suppose Assumption \ref{f_assump} holds. If the step sizes satisfy 
\begin{align}\label{step_range1}
    \eta_1\leq \frac{\lambda}{[(\lambda+1)(\kappa+1)+1]l}, \ \eta_2\leq \frac{1}{l}, \ \eta_1^2+\eta_2^2\leq \frac{1}{4l^2}.
\end{align}
Then for any integer $T>0$, the proposed Algorithm \ref{vr_rr} satisfies
\begin{align*}
    &\mathbb{E}_t\left[\mathcal{P}_{\lambda}(x_{t+1},y_{t+1})\right]- \mathcal{P}_{\lambda}(x_{t},y_{t})\\
    \leq & -C_1\|\nabla \Phi(x_t)\|^2 -C_2(\Phi(x_t) -f(x_t,y_t)),
\end{align*}
where $C_1:=\frac{\eta_1(\lambda-1)}{2}-8(\lambda\eta_1+\frac{\eta_1+\eta_2}{2})\eta_1^2l^2$ and $C_2:=\eta_2\mu-8(\lambda\eta_1+\frac{\eta_1+\eta_2}{2})(\mu\eta_2^2+2l\kappa\eta_1^2)l^2-2(\lambda+2)\eta_1 l\kappa$.
\end{lemma}
\begin{proof}
If $\eta_2\leq \frac{1}{l}$ and $\eta_1\leq \frac{\lambda}{l[(\lambda+1)(\kappa+1)+1]}$, then, the last two terms on the right-hand side of \eqref{Phi_ineq} are negative and can be eliminated. Then, taking the conditional expectation to Lemma \ref{Phi_dif_lem} and applying Lemma \ref{upp_bound_B}, we know that
\begin{align}\label{eq:suff-des-key}
    &\mathbb{E}_t\left[\mathcal{P}_{\lambda}(x_{t+1},y_{t+1})\right]-\mathcal{P}_{\lambda}(x_t,y_t)\notag\\
    \leq&-\frac{\eta_1(\lambda+1)}{2}\|\nabla \Phi(x_t)\|^2+2(\lambda+1)\eta_1 l\kappa  [\Phi(x_t)-\!f(x_t,y_t)]\notag\\
    &+\left(\frac{\eta_1}{2}+4\xi \eta_1^2l^2\right)\|\nabla_x f(x_t,y_t)\|^2\notag\\
    &-\left(\frac{\eta_2}{2}-4\xi\eta_2^2l^2\right)\|\nabla_y f(x_t,y_t)\|^2
\end{align}
where $\xi:=\lambda\eta_1+\frac{\eta_1+\eta_2}{2}$.
\par Since $f(x,\cdot)$ is $\mu$-strongly concave, we have
\begin{align*}
    -\|\nabla_y f(x_t,y_t)\|^2\leq -2\mu (\Phi(x_t)-f(x_t,y_t)).
\end{align*}
This together with \eqref{eq:suff-des-key} and the inequality 
\begin{align*}
    \|\nabla_x f(x_t,y_t)\|^2
    &\leq 2\|\nabla \Phi(x_t)\|^2\! +\! 2\|\nabla \Phi(x_t)-\!\nabla_x f(x_t,y_t)\|^2 \\
    &\leq 2\|\nabla \Phi(x_t)\|^2\! +\! 2l^2\|y^*(x_t)-y_t\|^2 \\
    &\leq  2\|\nabla \Phi(x_t)\|^2\!+\!4l\kappa [\Phi(x_t)-\!f(x_t,y_t)]
\end{align*}
implies that  
\begin{align*}
    &\mathbb{E}_t\left[\mathcal{P}_{\lambda}(x_{t+1},y_{t+1})\right]-\mathcal{P}_{\lambda}(x_t,y_t)\\
    \leq&-\left(\frac{\eta_1(\lambda+1)}{2}-\left(\eta_1+8\xi \eta_1^2l^2\right) \right)\|\nabla \Phi(x_t)\|^2\\
    &+2(\lambda+1)\eta_1 l\kappa  [\Phi(x_t)-f(x_t,y_t)]\\
    &+\left(\frac{\eta_1}{2}+4\xi \eta_1^2l^2\right)\cdot 4l\kappa [\Phi(x_t)-f(x_t,y_t)]\\
    &-\underbrace{\left(\eta_2-8\xi \eta_2^2 l^2\right)}_{C_0}\cdot \mu (\Phi(x_t)-f(x_t,y_t))\\
    =&-\underbrace{\left(\frac{\eta_1(\lambda-1)}{2}-8\xi \eta_1^2l^2\right) }_{C_1}\|\nabla \Phi(x_t)\|^2\\
    &-\underbrace{\left(\eta_2\mu-8\xi(\mu\eta_2^2+2l\kappa\eta_1^2)l^2-2(\lambda+2)\eta_1 l\kappa\right)}_{C_2} \\
    &\quad (\Phi(x_t)-f(x_t,y_t)).
\end{align*}
The desired results are derived.
\end{proof}

Now, we are ready to present the convergence result for the proposed Algorithm \ref{vr_rr}. We select appropriate step sizes and the parameter $\lambda$ to ensure that $\eta_1$ and $\eta_2$ not only satisfy the step size conditions \eqref{step_range1}, but also that the coefficient constants $C_1$ and $C_2$ are positive.

\begin{theorem} \label{con_theo}
Suppose Assumption \ref{f_assump} holds. If
$\lambda=4$, the step sizes $\eta_2\leq \frac{1}{8l}$ and $\eta_1=\frac{\eta_2}{r}$ with $r\geq 14\kappa^2$, 
then for any integer $T>0$, the proposed Algorithm \ref{vr_rr} achieves
\begin{align*}
    \frac{1}{T+1}\sum_{t=0}^T\mathbb{E}\left[\|\nabla \Phi(x_t)\|^2\right]\leq \frac{\mathcal{P}_{\lambda}(x_0,y_0)}{(T+1)\eta_1}=\mathcal{O}\left(\frac{r l}{T+1}\right).
\end{align*}    
\end{theorem}
\begin{proof}
Since $\eta_2\leq \frac{1}{8l}$, one has $\eta_1=\frac{\eta_2}{r}\leq \frac{1}{112 l \kappa^2}$. Then with $\kappa\geq 1$, the inequalities \eqref{step_range1} hold with $\lambda=4$. 
In the following part, we need to check $C_1,C_2>0$. \par First, applying $\lambda=4$, $\kappa\geq 1$, and $\eta_2/\eta_1=r\ge 14\kappa^2$, we derive that
    \begin{align*}
        \frac{C_1}{\eta_1}
        =\ &\frac{\lambda-1}{2}-8\left(\lambda\eta_1+\frac{\eta_1+\eta_2}{2}\right)\eta_1l^2\\
        =\ &\frac{3}{2}-8\left(\frac{9}{2}+\frac{r}{2}\right)\frac{\eta_2^2}{r^2}l^2\\
        \geq\ &\frac{3}{2}-\left(\frac{9}{2}+\frac{r}{2}\right)\frac{1}{8r^2}\\
        \geq\ & \frac{3}{2}-\left(\frac{9}{2}+\frac{14}{2}\right)\frac{1}{8\cdot14^2}>1.
    \end{align*}
    Then, using a similar derivation, we have
    \begin{align*}
        \frac{C_2}{\mu\eta_2}
        =\ &1-8\left(\lambda\eta_1+\frac{\eta_1+\eta_2}{2}\right)\left(\eta_2+\frac{2l\kappa\eta_1^2}{\mu\eta_2}\right)l^2\\
        &-\frac{2(\lambda+2)\eta_1 l\kappa}{\mu\eta_2}\\
        =\ & 1-8\left(\frac{9}{2r}+\frac{1}{2}\right)\left(1+\frac{2\kappa^2}{r^2}\right)\eta_2^2l^2-\frac{12\kappa^2}{r}\\
        \geq\ &1-\frac{1}{8}\left(\frac{9}{28\kappa^2}+\frac{1}{2}\right)\left(1+\frac{1}{14\cdot7\kappa^2}\right)-\frac{6}{7}\\
        \geq\ &\frac{1}{7}-\frac{1}{8}\left(\frac{9}{28}+\frac{1}{2}\right)\left(1+\frac{1}{14\cdot7}\right)> \frac{1}{26}.
    \end{align*}
    Hence, $C_1>\eta_1$ and $C_2>\frac{\mu \eta_2}{26}=\frac{\eta_1 r \mu}{26}=\frac{14 \eta_1 l\kappa}{26}>\frac{\eta_1 l \kappa}{2}$. Then, the potential function satisfies
    \begin{align}\label{P_lam_in}
     &\mathbb{E}_t\left[\mathcal{P}_{\lambda}(x_{t+1},y_{t+1})\right]-\mathcal{P}_{\lambda}(x_t,y_t)\notag\\
     \leq &-\eta_1 \|\nabla \Phi(x_t)\|^2-\frac{\eta_1 l \kappa}{2}(\Phi(x_t)-f(x_t,y_t)).
    \end{align}
    By taking expectation to both sides of \eqref{P_lam_in} and $\Phi(x)\geq f(x,y)$ for any $(x,y)$, it follows that
    \begin{align*}
        \mathbb{E}\left[\mathcal{P}_{\lambda}(x_{t+1},y_{t+1})-\mathcal{P}_{\lambda}(x_t,y_t)\right]\leq -\eta_1 \mathbb{E}\left[\|\nabla \Phi(x_t)\|^2\right].
    \end{align*}
    Taking the average over $t=0,\ldots, T$ and rearranging the terms, then we have
    \begin{align*}
        \frac{1}{T+1} \sum_{t=0}^T\mathbb{E}\left[\|\nabla \Phi(x_t)\|^2\right]\leq \frac{\mathcal{P}_{\lambda}(x_0,y_0)}{(T+1)\eta_1},
    \end{align*}
    which is the desired result.
\end{proof}
\begin{figure*}[t]
	\centering
	\subfigure[Game-stationarity measure $\|\nabla f(x_t,\theta_t)\|$]{
		\includegraphics[width=8cm]{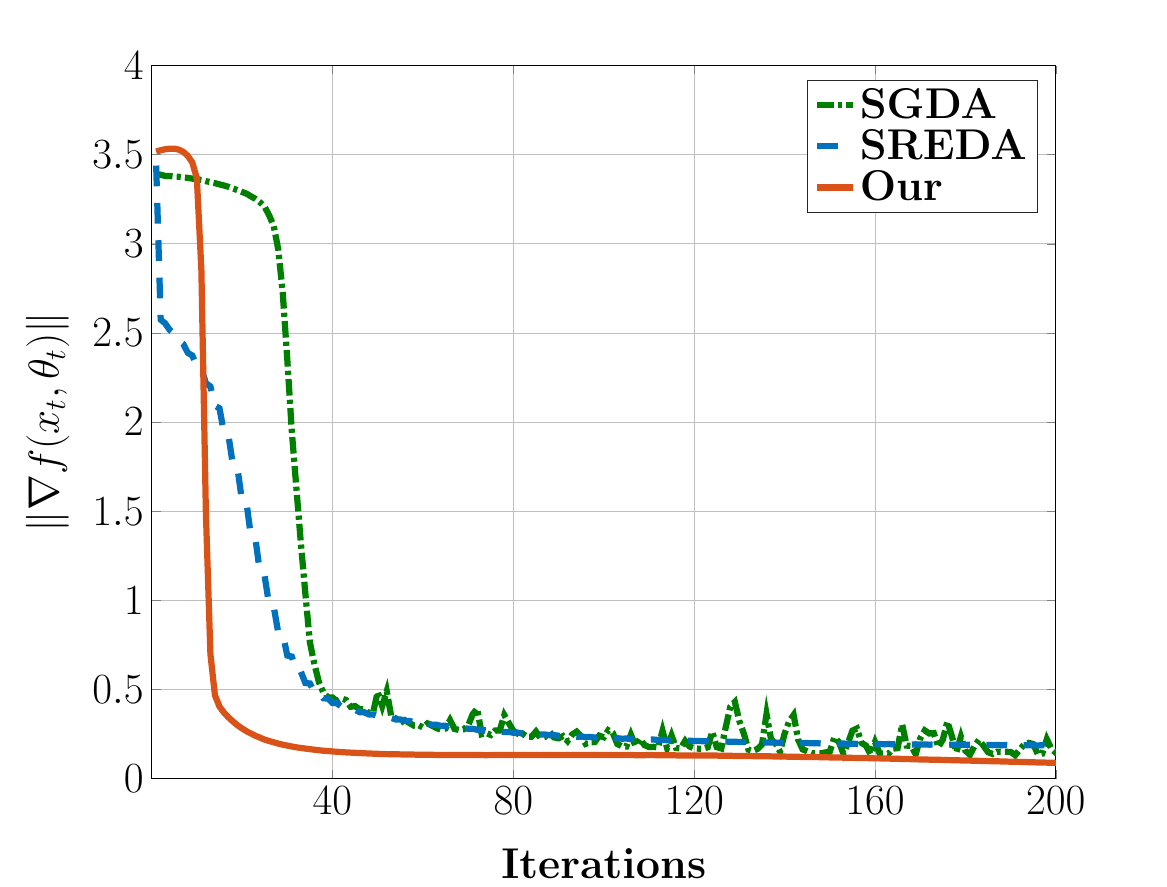}
        \label{poison_grad_fig}
	}\qquad
	\subfigure[Training accuracy]{
		\includegraphics[width=8cm]{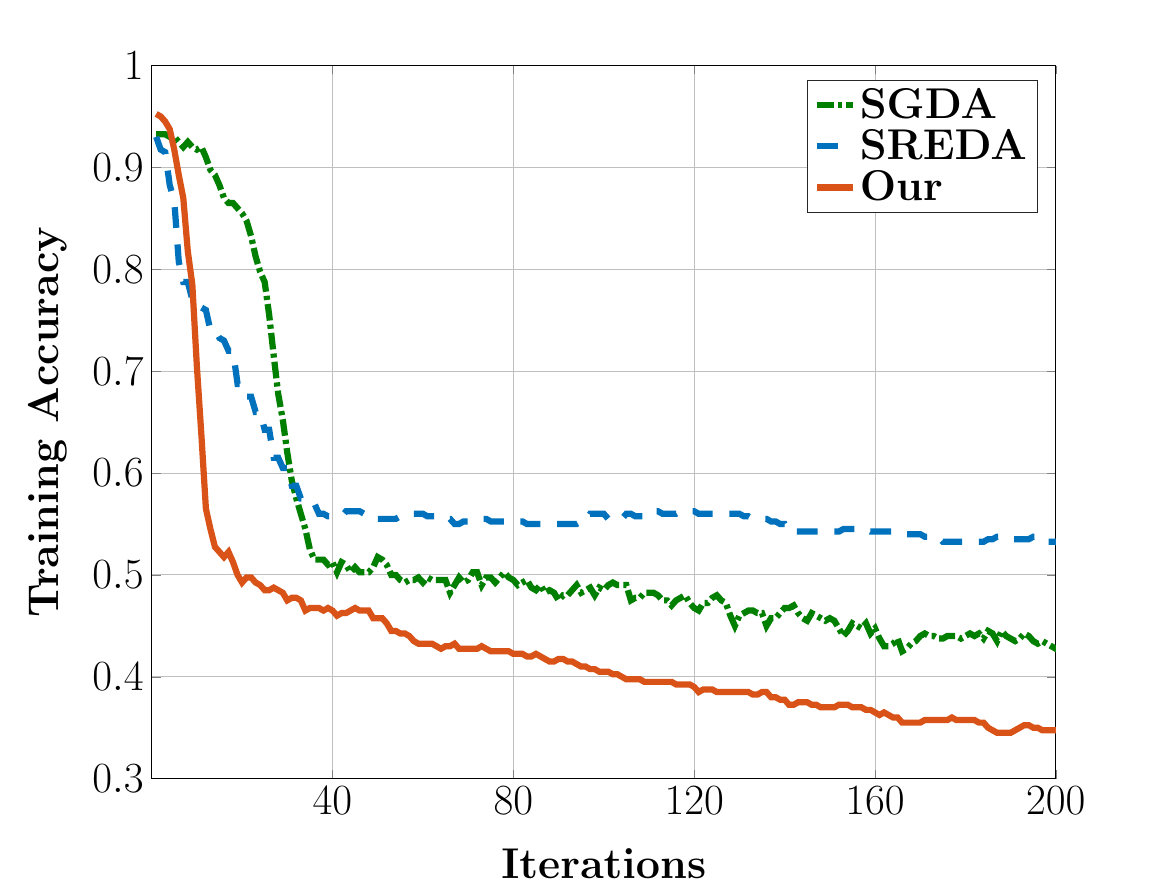}
        \label{accu_fig}
	}
	\caption{Iterative performance of SGDA, SREDA and Algorithm \ref{vr_rr} in data poisoning.}
 \label{poison_fig}
\end{figure*}

The following corollary directly follows from Theorem \ref{con_theo}.

\begin{corollary}
Suppose assumptions in Theorem \ref{con_theo} are satisfied with the step sizes 
$\eta_1=\mathcal{O}(\frac{1}{\kappa^2l})$ and $\eta_2=\mathcal{O}(\frac{1}{l})$, then  Algorithm \ref{vr_rr} achieves $\epsilon$-stationary points with $\mathcal{O}(\kappa^2\epsilon^{-2})$ iterations and $\mathcal{O}(n\kappa^2\epsilon^{-2})$ gradient oracles. 
\end{corollary}

\begin{remark}
    According to the step size setting in Theorem \ref{con_theo}, the ratio of step sizes $\eta_2/\eta_1$ needs to be $\Theta(\kappa^2)$. This requirement arises due to the non-symmetric nature of the NC-SC minimax optimization, which employs a two-timescale scheme to balance primal and dual updates. This approach is easier to implement than multi-step stochastic algorithms. For instance, the SREDA algorithm \cite{SREDA_nips2020} requires a step size $\eta_1$ proportional to the accuracy $\epsilon$. As a result, if the accuracy is small, the step size can also be very small.
\end{remark}
\begin{remark}
    
    For NC-SC minimax problems, the uniform sampling SGDA algorithm in \cite{SGDA} achieves a convergence rate of $\mathcal{O}\left(\frac{1}{T+1}\right) + \frac{\kappa \sigma^2}{M}$, where $\sigma$ is the upper bound of gradient variance and $M = \Theta(\max\{1, \kappa \sigma^2 \epsilon^{-2}\})$ is the necessary batch size for the desired convergence property. 
    In contrast, our proposed Algorithm \ref{vr_rr} achieves a convergence rate without the variance term. Additionally, the SGDA with RR in \cite{SGDA_ICLR_NCPL} achieves $\epsilon$-stationary points within $\mathcal{O}(\sqrt{n}\kappa^3  \epsilon^{-3})$ gradient oracles. In comparison, our proposed algorithm requires $\mathcal{O}(n \kappa^2 \epsilon^{-2})$ gradient oracles, outperforming \cite{SGDA_ICLR_NCPL} when $\epsilon\leq \mathcal{O}(\kappa/\sqrt{n})$.
\end{remark}

\section{Numerical Experiments}\label{simulation}

In this section, we demonstrate the practical efficacy of the proposed algorithm by applying it to data poisoning and distributionally robust optimization (DRO). Data poisoning is an adversarial attack where attackers manipulate the training dataset to degrade prediction accuracy. DRO is a fundamental problem formulation used in numerous control and signal processing applications, such as power control and transceiver design problems \cite{Lu_2020}. We compare the performance of the random reshuffled Algorithm \ref{vr_rr} with the widely-used SGDA algorithm \cite{SGDA} and the SREDA algorithm presented in \cite{SREDA_nips2020}. In all experiments, the parameters of the
compared algorithms are set as specified in their original literature.


\subsection{Data Poisoning against Logistic Regression}

Let $\mathcal{D} = \{z_i, t_i\}_{i=1}^n$ denote the training dataset, where $n' \ll n$ samples are corrupted by a perturbation vector $x$, resulting in poisoned training data $z_i + x$, which disrupts the training process and reduces prediction accuracy. Following the setup in \cite{xu2024derivativefree,pmlr-v119-liu20j}, we generate a dataset containing $n = 1000$ samples $\{z_i, t_i\}_{i=1}^n$, where $z_i \in \mathbb{R}^{100}$ are sampled from $\mathcal{N}(\mathbf{0}, \mathbf{I})$. With $\nu_i \sim \mathcal{N}(0, 10^{-3})$, we set
\[
t_i = \begin{cases} 
1, & \text{if} \; 1/(1+e^{-z_i^T \theta^*+\nu_i}) > 0.5, \\
0, & \text{otherwise}.
\end{cases}
\]
Here, we choose $\theta^*$ as the base model parameters. The dataset is randomly split into a training dataset $\mathcal{D}_{\text{train}}$ and a testing dataset $\mathcal{D}_{\text{test}}$. The training dataset is further divided into $\mathcal{D}_{\operatorname{tr}, 1}$ and $\mathcal{D}_{\operatorname{tr}, 2}$, where $\mathcal{D}_{\operatorname{tr}, 1}$ (resp. $\mathcal{D}_{\operatorname{tr}, 2}$) represents the poisoned (resp. unpoisoned) subset of the training dataset. The problem can then be formulated as follows:
\begin{align*}
    \underset{\|x\|_{\infty} \leq \epsilon}{\max} \underset{\theta}{\min}\, f(x, \theta;\mathcal{D}_{\text{train}})
\end{align*}
where $f(x, \theta ; \mathcal{D}_{\text{train}}):=F(x, \theta ; \mathcal{D}_{\operatorname{tr}, 1})+F(0, \theta ; \mathcal{D}_{\operatorname{tr}, 2})$ with 
\begin{align*}
F(x, \theta;\mathcal{D})
=&-\frac{1}{|\mathcal{D}|} \sum_{(z_i,t_i) \in \mathcal{D}}[t_i \log (l(x,\theta;z_i))\\
&+(1-t_i) \log (1-l(x, \theta;z_i))]
\end{align*}
and $l(x, \theta ; z_i)=1/(1+e^{-(z_i+x)^T \theta})$ taking a logistic regression model. 


We set the poisoning ratio $|\mathcal{D}_{\text{tr},1}|/|\mathcal{D}_{\text{train}}| = 10\%$, $\epsilon = 2$, and the number of iterations $T = 200$. For each algorithm, we perform a grid search for the learning rates $\eta_1$ and $\eta_2$ from the set $\{0.1, 0.01, 0.001\}$. To compare the performance of different algorithms, we evaluate their efficiency using the game-stationarity gap $\|\nabla f(x_t, \theta_t)\|$ as used in literature \cite{xu2024derivativefree,pmlr-v119-liu20j}. 
Fig. \ref{poison_fig}(a) shows the trajectories of the stationary gap for the comparative algorithms. It is observed that the norm of the stationary gap for Algorithm \ref{vr_rr} decreases faster than that of the other two algorithms. Fig. \ref{poison_fig}(b) presents the prediction accuracy of different algorithms over the testing dataset. We observe that the proposed Algorithm \ref{vr_rr} yields lower prediction accuracy compared to the other algorithms, indicating that Algorithm \ref{vr_rr} achieves superior attack performance.


\subsection{Distributionally Robust Optimization}

We consider a nonconvex-regularized variant of the DRO problem, which arises in distributionally robust learning. Given a dataset $\left\{\left(z_i, t_i\right)\right\}_{i=1}^n$, where $z_i \in \mathbb{R}^d$ represents the features and $t_i \in \{-1,1\}$ represents the labels, the robust logistic regression problem is formulated as follows:
$$
\min_{x \in \mathbb{R}^d} \max_{y \in \Delta_n} f(x, y) = \sum_{i=1}^n y_i l_i(x) - V(y) + g(x),
$$
where $\Delta_n$ denotes the simplex in $\mathbb{R}^n$, $y_i$ is the $i$-th component of the variable $y$, and the  logistic loss function $l_i(x) = \log(1 + e^{-t_i z_i^\top x})$. The function $V$ is a divergence measure defined as $V(y) = \frac{1}{2} \lambda_1 \|ny - \mathbf{1}\|^2$ and the nonconvex regularization $g$ is given by $g(x) = \lambda_2 \sum_{i=1}^d \alpha x_i^2/(1 + \alpha x_i^2)$. Following the settings in \cite{SREDA_nips2020, den_nonmini_nips}, we set $\lambda_1 = 1/n^2$, $\lambda_2 = 0.001$, and $\alpha = 10$ in our experiment.

\begin{figure}[h]
    \centering
    \includegraphics[width=8cm]{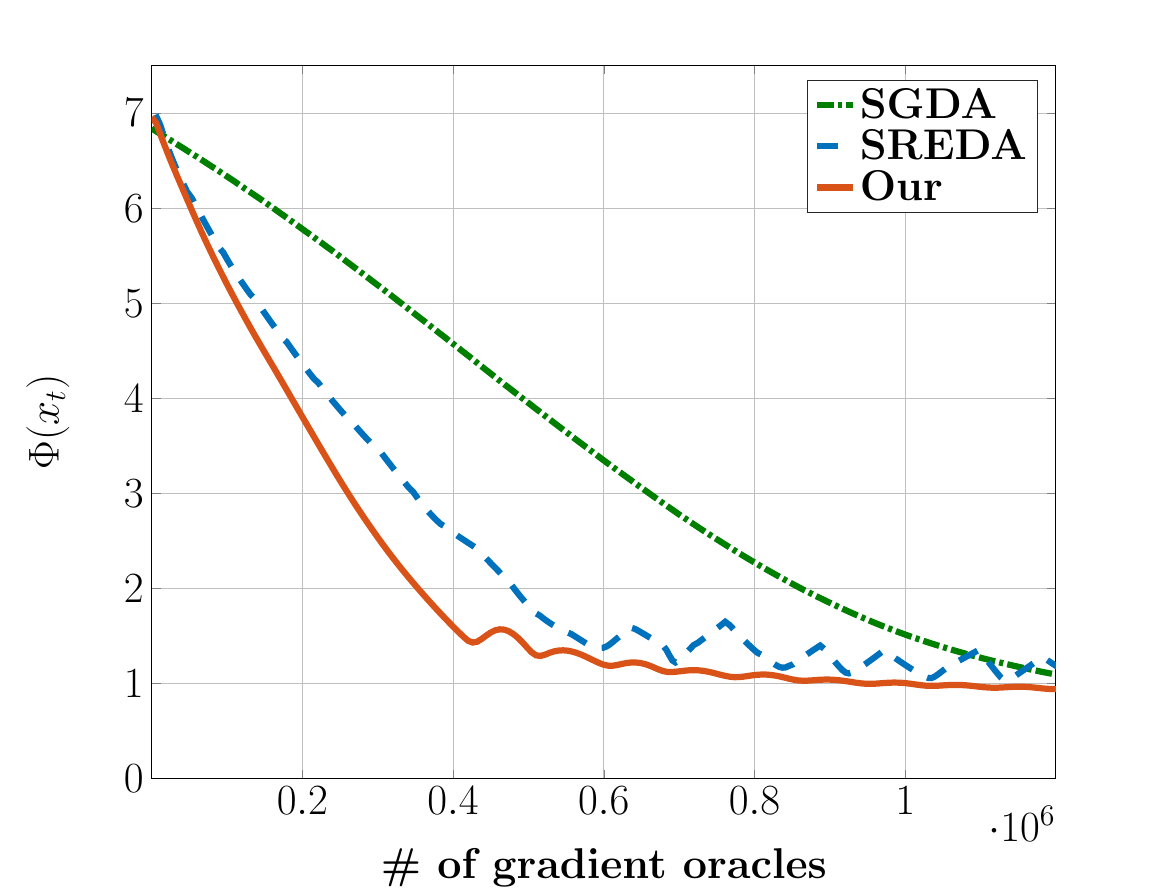}
    \caption{Performance of $\Phi(x_t)$ with respect to the number of gradient oracles  for algorithms in robust logistic regression.}
    \label{RLR_a9a_fig}
\end{figure}

The experiment is conducted on the real-world dataset a9a, where $d = 123$ and $n = 32,561$. We evaluate the  value of $\Phi(x)$ (defined in \eqref{min_opti}) with respect to the number of gradient oracles to compare the performance of different algorithms. From the convergence trajectories shown in Fig. \ref{RLR_a9a_fig}, we observe that the proposed Algorithm \ref{vr_rr} converges faster than other baseline algorithms, verifying its superior performance.

\section{Conclusion}\label{conclusion}

For nonconvex-strongly concave minimax problems, this paper developed a variance-reduced shuffling gradient algorithm with various shuffling schemes. The proposed algorithm mitigates the impact of gradient variance present in most existing stochastic algorithms and achieves a better gradient complexity compared to existing shuffling algorithms for minimax optimization. The complexity of the proposed stochastic algorithm matches the tight complexity performance achieved by its deterministic counterpart, showcasing its efficiency and effectiveness in addressing minimax problems. 
Further research direction could be the development of sampling-without-replacement algorithms for more general nonconvex-concave finite-sum minimax problems.

\bibliographystyle{IEEEtran}
\bibliography{refer}

\begin{thebibliography}{10}
\providecommand{\url}[1]{#1}
\csname url@samestyle\endcsname
\providecommand{\newblock}{\relax}
\providecommand{\bibinfo}[2]{#2}
\providecommand{\BIBentrySTDinterwordspacing}{\spaceskip=0pt\relax}
\providecommand{\BIBentryALTinterwordstretchfactor}{4}
\providecommand{\BIBentryALTinterwordspacing}{\spaceskip=\fontdimen2\font plus
\BIBentryALTinterwordstretchfactor\fontdimen3\font minus
  \fontdimen4\font\relax}
\providecommand{\BIBforeignlanguage}[2]{{%
\expandafter\ifx\csname l@#1\endcsname\relax
\typeout{** WARNING: IEEEtran.bst: No hyphenation pattern has been}%
\typeout{** loaded for the language `#1'. Using the pattern for}%
\typeout{** the default language instead.}%
\else
\language=\csname l@#1\endcsname
\fi
#2}}
\providecommand{\BIBdecl}{\relax}
\BIBdecl

\bibitem{NIPS2014_GAN}
I.~Goodfellow, J.~Pouget-Abadie, M.~Mirza, B.~Xu, D.~Warde-Farley, S.~Ozair,
  A.~Courville, and Y.~Bengio, ``Generative adversarial nets,'' in
  \emph{Advances in Neural Information Processing Systems}, vol.~27, 2014.

\bibitem{RL_saddle}
K.~Zhang, S.~Kakade, T.~Basar, and L.~Yang, ``Model-based multi-agent {RL} in
  zero-sum {M}arkov games with near-optimal sample complexity,'' in
  \emph{Advances in Neural Information Processing Systems}, vol.~33, 2020, pp.
  1166--1178.

\bibitem{robust_control}
X.~Qiu, J.-X. Xu, Y.~Xu, and K.~C. Tan, ``A new differential evolution
  algorithm for minimax optimization in robust design,'' \emph{IEEE
  Transactions on Cybernetics}, vol.~48, no.~5, pp. 1355--1368, 2018.

\bibitem{source_allo_conti}
J.~W. Simpson-Porco, B.~K. Poolla, N.~Monshizadeh, and F.~Dörfler,
  ``Input–output performance of linear–quadratic saddle-point algorithms
  with application to distributed resource allocation problems,'' \emph{IEEE
  Transactions on Automatic Control}, vol.~65, no.~5, pp. 2032--2045, 2020.

\bibitem{mini_model_pre}
A.~Pavlov, I.~Shames, and C.~Manzie, ``Minimax strategy in approximate model
  predictive control,'' \emph{Automatica}, vol. 111, p. 108649, 2020.

\bibitem{unified_dis_tac}
Y.~Huang, Z.~Meng, J.~Sun, and W.~Ren, ``A unified distributed method for
  constrained networked optimization via saddle-point dynamics,'' \emph{IEEE
  Transactions on Automatic Control}, vol.~69, no.~3, pp. 1818--1825, 2024.

\bibitem{pmlr-v89-du19b}
S.~S. Du and W.~Hu, ``Linear convergence of the primal-dual gradient method for
  convex-concave saddle point problems without strong convexity,'' in
  \emph{Proceedings of the Twenty-Second International Conference on Artificial
  Intelligence and Statistics}, vol.~89, 2019, pp. 196--205.

\bibitem{DPD_aggre_game}
Y.~Wang and A.~Nedić, ``Differentially private distributed algorithms for
  aggregative games with guaranteed convergence,'' \emph{IEEE Transactions on
  Automatic Control}, vol.~69, no.~8, pp. 5168--5183, 2024.

\bibitem{Lu_2020}
S.~Lu, I.~Tsaknakis, M.~Hong, and Y.~Chen, ``Hybrid block successive
  approximation for one-sided non-convex min-max problems: Algorithms and
  applications,'' \emph{IEEE Transactions on Signal Processing}, vol.~68, pp.
  3676--3691, 2020.

\bibitem{pmlr-v80-dai18c}
B.~Dai, A.~Shaw, L.~Li, L.~Xiao, N.~He, Z.~Liu, J.~Chen, and L.~Song,
  ``{SBEED}: Convergent {R}einforcement learning with nonlinear function
  approximation,'' in \emph{Proceedings of the 35th International Conference on
  Machine Learning}, vol.~80, 2018, pp. 1125--1134.

\bibitem{pmlr-v119-liu20j}
S.~Liu, S.~Lu, X.~Chen, Y.~Feng, K.~Xu, A.~Al-Dujaili, M.~Hong, and U.-M.
  O'Reilly, ``Min-max optimization without gradients: Convergence and
  applications to black-box evasion and poisoning attacks,'' in
  \emph{Proceedings of the 37th International Conference on Machine Learning},
  vol. 119, 2020, pp. 6282--6293.

\bibitem{SREDA_nips2020}
L.~Luo, H.~Ye, Z.~Huang, and T.~Zhang, ``Stochastic recursive gradient descent
  ascent for stochastic nonconvex-strongly-concave minimax problems,'' in
  \emph{Advances in Neural Information Processing Systems}, vol.~33, 2020, pp.
  20\,566--20\,577.

\bibitem{SGDA}
T.~Lin, C.~Jin, and M.~Jordan, ``On gradient descent ascent for
  nonconvex-concave minimax problems,'' in \emph{Proceedings of the 37th
  International Conference on Machine Learning}, vol. 119, 2020, pp.
  6083--6093.

\bibitem{two_GDA}
M.~Heusel, H.~Ramsauer, T.~Unterthiner, B.~Nessler, and S.~Hochreiter, ``{GAN}s
  trained by a two time-scale update rule converge to a local {N}ash
  equilibrium,'' in \emph{Advances in Neural Information Processing Systems},
  vol.~30, 2017.

\bibitem{SGD_max}
C.~Jin, P.~Netrapalli, and M.~Jordan, ``What is local optimality in
  nonconvex-nonconcave minimax optimization?'' in \emph{Proceedings of the 37th
  International Conference on Machine Learning}, vol. 119, 2020, pp.
  4880--4889.

\bibitem{NC_if}
M.~Nouiehed, M.~Sanjabi, T.~Huang, J.~D. Lee, and M.~Razaviyayn, ``Solving a
  class of non-convex min-max games using iterative first order methods,'' in
  \emph{Advances in Neural Information Processing Systems}, vol.~32, 2019.

\bibitem{why_rand_good}
M.~Gurbuzbalaban, A.~Ozdaglar, and P.~A. Parrilo, ``Why random reshuffling
  beats stochastic gradient descent,'' \emph{Mathematical Programming}, vol.
  186, pp. 49--84, 2021.

\bibitem{RR_improv}
K.~Mishchenko, A.~Khaled, and P.~Richtarik, ``Random reshuffling: Simple
  analysis with vast improvements,'' in \emph{Advances in Neural Information
  Processing Systems}, vol.~33, 2020, pp. 17\,309--17\,320.

\bibitem{SGD_opti_rate}
K.~Ahn, C.~Yun, and S.~Sra, ``{SGD} with shuffling: {O}ptimal rates without
  component convexity and large epoch requirements,'' in \emph{Advances in
  Neural Information Processing Systems}, 2020.

\bibitem{SGDA_ICLR_NCPL}
H.~Cho and C.~Yun, ``{SGDA} with shuffling: {F}aster convergence for
  nonconvex-{P{\L}} minimax optimization,'' in \emph{the Eleventh International
  Conference on Learning Representations}, 2023.

\bibitem{VR_CC}
B.~Palaniappan and F.~Bach, ``Stochastic variance reduction methods for
  saddle-point problems,'' in \emph{Advances in Neural Information Processing
  Systems}, vol.~29, 2016.

\bibitem{GAN_VR_nips}
T.~Chavdarova, G.~Gidel, F.~Fleuret, and S.~Lacoste-Julien, ``Reducing noise in
  {GAN} training with variance reduced extragradient,'' in \emph{Advances in
  Neural Information Processing Systems}, vol.~32, 2019.

\bibitem{comp_vr_finite}
Y.~Han, G.~Xie, and Z.~Zhang, ``Lower complexity bounds of finite-sum
  optimization problems: The results and construction,'' \emph{Journal of
  Machine Learning Research}, vol.~25, no.~2, pp. 1--86, 2024.

\bibitem{chen2023faster}
L.~Chen, B.~Yao, and L.~Luo, ``Faster stochastic algorithms for minimax
  optimization under {P}olyak-{{\L}}ojasiewicz condition,'' in \emph{Advances
  in Neural Information Processing Systems}, vol.~35, 2022, pp.
  13\,921--13\,932.

\bibitem{NIPS_VR_PL}
J.~Yang, N.~Kiyavash, and N.~He, ``Global convergence and variance reduction
  for a class of nonconvex-nonconcave minimax problems,'' in \emph{Advances in
  Neural Information Processing Systems}, vol.~33, 2020, pp. 1153--1165.

\bibitem{NIPS_SAPD}
X.~Zhang, N.~S. Aybat, and M.~Gurbuzbalaban, ``{SAPD}+: An accelerated
  stochastic method for nonconvex-concave minimax problems,'' in \emph{Advances
  in Neural Information Processing Systems}, vol.~35, 2022, pp.
  21\,668--21\,681.

\bibitem{vr_vi_cui}
S.~Cui and U.~Shanbhag, ``On the analysis of variance-reduced and randomized
  projection variants of single projection schemes for monotone stochastic
  variational inequality problems,'' \emph{Set-Valued and Variational
  Analysis}, vol.~29, p. 453–499, 2021.

\bibitem{VR_mon_general}
S.~Cui and U.~V. Shanbhag, ``Variance-reduced splitting schemes for monotone
  stochastic generalized equations,'' \emph{IEEE Transactions on Automatic
  Control}, vol.~68, no.~11, pp. 6636--6648, 2023.

\bibitem{SVRG_nips}
R.~Johnson and T.~Zhang, ``Accelerating stochastic gradient descent using
  predictive variance reduction,'' in \emph{Advances in Neural Information
  Processing Systems}, vol.~26, 2013.

\bibitem{xu2024derivativefree}
Z.~Xu, Z.~Wang, J.~Shen, and Y.~Dai, ``Derivative-free alternating projection
  algorithms for general nonconvex-concave minimax problems,''
  \emph{arXiv.2108.00473}, 2024.

\bibitem{den_nonmini_nips}
W.~Xian, F.~Huang, Y.~Zhang, and H.~Huang, ``A faster decentralized algorithm
  for nonconvex minimax problems,'' in \emph{Advances in Neural Information
  Processing Systems}, vol.~34, 2021, pp. 25\,865--25\,877.

\end{thebibliography}
\end{document}